\documentclass[11pt,reqno]{amsart}
\usepackage{amsmath,amsthm,amssymb,mathrsfs,stmaryrd,color, amsfonts}
\usepackage[all]{xy}
\usepackage{url}
\numberwithin{equation}{section}
\usepackage[utf8]{inputenc}
\usepackage[T1]{fontenc}

\usepackage{extarrows}
\usepackage{mathtools}
\usepackage{enumerate}
\usepackage{makecell}

\usepackage[colorlinks=true,hyperindex, linkcolor=blue, pagebackref=false, citecolor=blue,pdfpagelabels]{hyperref}
\usepackage[capitalize]{cleveref}

\newtheorem{theorem}{Theorem}[section]
\newtheorem*{theorem*}{Theorem}
\newtheorem*{definition*}{Definition}

\newtheorem{theo}[theorem]{Theorem}
\newtheorem*{theo*}{Theorem}
\newtheorem{lem}[theorem]{Lemma}
\newtheorem{prop}[theorem]{Proposition}
\newtheorem{cor}[theorem]{Corollary}

\theoremstyle{definition}

\newtheorem{remark}[theorem]{Remark}

\newtheorem{example}[theorem]{Example}

\newtheorem{rem}[theorem]{Remark}
\newtheorem{expl}[theorem]{Example}

\DeclareMathOperator{\Br}{Br}

\DeclareMathOperator{\BN}{BN}
\DeclareMathOperator{\codim}{codim}
\DeclareMathOperator{\CH}{CH}

\DeclareMathOperator{\ch}{ch}

\DeclareMathOperator{\Ext}{Ext}

\DeclareMathOperator{\id}{id}
\DeclareMathOperator{\im}{im}

\DeclareMathOperator{\Pic}{Pic}

\DeclareMathOperator{\rk}{rk}

\DeclareMathOperator{\Supp}{Supp}

\DeclareMathOperator{\Sym}{Sym}

\DeclareMathOperator{\sheafExt}{\mathcal{E}\it{xt}}

\renewcommand{\epsilon}{\varepsilon}

\newcommand{\Ccal}{\mathcal{C}}

\newcommand{\Ecal}{\mathcal{E}}
\newcommand{\Fcal}{\mathcal{F}}

\newcommand{\Ical}{\mathcal{I}}
\newcommand{\Lcal}{\mathcal{L}}
\newcommand{\Mcal}{\mathcal{M}}

\newcommand{\Ocal}{\mathcal{O}}

\newcommand{\Qcal}{\mathcal{Q}}

\newcommand{\Tcal}{\mathcal{T}}

\newcommand{\Z}{\mathbb{Z}}
\newcommand{\C}{\mathbb{C}}

\newcommand{\Abb}{\mathbb{A}}

\newcommand{\Gbb}{\mathbb{G}}
\newcommand{\Pbb}{\mathbb{P}}

\usepackage{color}
\newcounter{commentcounter}

\def\?{\ {\bf\color{red}???}\ 
\immediate\write16{}
\immediate\write16{Warning: There was still a question mark . . . }
\immediate\write16{}}

\begin{document}

\title[Constant cycle subvarieties in the Mukai system of rank two and genus two]{Constant cycle subvarieties in the Mukai system of rank two and genus two}
	\author[I.~Hellmann]{Isabell Hellmann}
	\thanks{The author is supported by the SFB/TR 45 `Periods, Moduli Spaces and Arithmetic of Algebraic varieties' of the DFG (German Research Foundation) and the Bonn International Graduate School.}
	\address{Mathematisches Institut, Universit\"at Bonn, Endenicher Allee 60, 53115 Bonn, Germany}
	%\address{Mathematisches Institut, Universit\"at Bonn}
	\email{igb@math.uni-bonn.de}
	
	\begin{abstract}\noindent
	Combining theorems of Voisin and Marian, Shen, Yin and Zhao, we compute the dimensions of the orbits under rational equivalence in the Mukai system of rank two and genus two. We produce several examples of algebraically coisotropic and constant cycle subvarieties.
	%\vspace{-2mm}
	\end{abstract}
	
	\maketitle
	\setcounter{tocdepth}{2}
	
\section{Introduction}
By a theorem of Beauville and Voisin \cite{BV}, any point lying on a rational curve in a K3 surface $S$, determines the same zero cycle of degree one
\[ c_S \in \CH_0(S),\]
called the \emph{Beauville--Voisin class}. This class has the striking property that the image of the intersection product
\[ \Pic(S) \otimes \Pic(S) \rightarrow \CH_0(S)\]
and $c_2(S)$ are contained in $\Z \cdot c_S$. It is expected that the Chow ring of an irreducible holomorphic symplectic manifold has a similar and particularly rich structure. The goal of this note is to investigate the Chow group of zero cycles for the Mukai system of rank two and genus two. Specifically, we produce examples of constant cycle and algebraically coisotropic subvarieties.\\

Let $(S,H)$ be a polarized K3 surface of genus $2$, that is a double covering $\pi \colon S \rightarrow \Pbb^2$ ramified over a sextic curve and $H = \pi^*\Ocal(1)$ is primitive. We consider the moduli space $M=M_H(0,2H,s)$ of $H$-Gieseker stable coherent sheaves on $S$ with Mukai vector $v=(0,2H,s)$ where $s \equiv 1$ mod $2$. This is an irreducible holomorphic symplectic variety, which is birational to $S^{[5]}$. A point in $M_H(0,2H,s)$ corresponds to a stable sheaf $\Ecal$ on $S$ such that $\Ecal$ is pure of dimension one with support in the linear system $|2H|$ and $\chi(\Ecal)=s$. Taking the (Fitting) support defines a Lagrangian fibration
\[ f \colon M_H(0,2H,s) \longrightarrow B  \coloneqq |2H| \cong \Pbb^{5} \]
known as the \emph{Mukai system of rank two and genus two} \cite{Beau}, \cite{Mu}. It enjoys many beautiful features and is studied from various perspectives. For example, one can view it as a compactified relative Jacobian, as a generalisation of the Hitchin system \cite{DEL} or as the birational model of $S^{[5]}$ admitting a Lagrangian fibration.\\

For any irreducible, holomorphic symplectic manifold $X$ of dimension $2n$, a brute force approach to finding constant cycle subvarieties (see Section \ref{subsection preliminaries} for the definition) is to consider the orbit under rational equivalence of a point $x \in X$. This is the countable union of algebraic subvarieties defined by
\[ O_x \coloneqq \{ x' \in X \mid [x] = [x'] \in \CH_0(X)\} \subset X. \]
Then $\dim O_x$ is defined to be the supremum over the dimensions of the components of $O_x$. In \cite{Voi}, Voisin defines an increasing filtration $F_0X \subset F_1X \subset \ldots \subset F_nX =X$ on the points of $X$, where
\[ F_iX \coloneqq \{ x \in X \mid \dim O_x \geq n-i \} \]
is again a countable union of algebraic subvarieties. Our examples are based on the combination of two theorems. The first one is due to Voisin.

\begin{theo}[{\cite[Thm 1.3]{Voi}}]\label{theo1}
We have $\dim F_i X \leq n + i$ and if $Z \subset F_i X$ is an irreducible component of dimension $n+i$. Then $Z$ is algebraically coisotropic and the fibers of the isotropic fibration are constant cycle subvarieties of dimension $n-i$.
\end{theo}

The second theorem applies in the case that $X = M_\sigma(v)$ is a smooth projective moduli of Bridgeland stable objects in $D^b(S)$ and is due to Marian, Shen, Yin and Zhao. It establishes a link between rational equivalence in $X$ and in $S$, which in particular results in a connection between Voisin's filtration $F_\bullet X$ and O'Grady's filtration $S_\bullet\CH_0(S)$.

\begin{theo}[{\cite{SYZ},\cite{MZ}}, Thms \ref{syz}, \ref{filtrations are equal}]\label{theo2}
\begin{enumerate}[\rm(i)]
\item Any two points $\Ecal, \Ecal' \in M_\sigma(v)$ are rational equivalent in $M_\sigma(v)$ if and only if $\ch_2(\Ecal') = \ch_2(\Ecal) \in \CH_0(S).$
%\item  $[\Ecal] = [\Ecal'] \in  \CH_0(X) \iff \ch_2(\Ecal') = \ch_2(\Ecal) \in \CH_0(S). $
%\item  $\ch_2(\Ecal) \in S_i\CH_0(S) \Rightarrow \Ecal \in F_iX$\\
\item Let $\Ecal \in M_\sigma(v)$ such that $\ch_2(\Ecal) \in S_i\CH_0(S)$. Then $\Ecal \in F_iM_\sigma(v)$. If  $M_\sigma(v)$ is birational to the Hilbert scheme $S^{[n]}$, then also the converse implication holds true, i.e.\ in this case
\[ F_iM_\sigma(v) = \{ \Ecal \in M_\sigma(v) \mid \ch_2(\Ecal) \in  S_i\CH_0(S) \}.\]
\end{enumerate}
\end{theo}

We remark that both parts of the theorem can equally be formulated with $c_2$ instead of $\ch_2$.\\

This opens the door to finding infinitely many examples of constant cycle or algebraically coisotropic subvarieties in $M= M_H(0,2H,-1)$. For example, a first straightforward application yields.

\begin{lem}[Cor \ref{fiber ccL}]
The fiber $F= f^{-1}(D)$ over $D \in |2H|$ is a constant cycle Lagrangian if and only if $D \subset S$ is a constant cycle curve.
\end{lem}

Or one can prove, that given $\Ecal \in M$ such that $\Supp(\Ecal) = D$. Then $\Ecal \in F_{g(\tilde{D})}M$, where $g(\tilde{D})$ is the geometric genus of $D$. This way, we find algebraically coisotropic subvarieties over singular curves. Precisely, for $i=0,\ldots,4$ let
\[ V_i \coloneqq \{ D \in |2H| \mid g(\tilde{D}) \leq i \} \subset |2H| \]
and set $M_{V_i} \coloneqq f^{-1}(V_i)$.

\begin{prop}[Prop \ref{vertical examples from sing curves}]
The subvarieties $M_{V_i}$ are equidimensional of codimension $n-i$ and satisfy
\[ M_{V_i} \subset F_i M. \]
In particular, they are algebraically coisotropic.
\end{prop}
Actually, $V_i$ is reducible due to reducible and non-reduced curves in the linear system $|2H|$. For every component we find the isotropic fibration and comment on the resulting constant cycle subvarieties. Most of them are rational. However, over the component of non-reduced curves $\Delta \subset V_2$, we find three-dimensional constant cycle subvarieties that are not rational (cf.\ Proposition \ref{T Delta}).\\

Another series of examples comes from Brill--Noether theory. Let $B^\circ \subset B$ be the locus of smooth curves and $\Ccal^\circ \rightarrow B^\circ$ the restricted universal curve. For any $i$, we have an isomorphism
\[ M_H(0,2H,i-4)^\circ  \cong \Pic^i_{\Ccal^\circ / B^\circ},\]
where $M_H(0,2H,i-4)^\circ$ is the preimage of $B^\circ$ under the support map $M_H(0,2H,i-4) \rightarrow B$. For $i=1,\ldots 4$, we define
\[ \BN^0_i(B^\circ) \coloneqq \{ \Lcal \in M_H(0,2H,i-4)^\circ\mid H^0(S,\Lcal) \neq 0\} \subset M_H(0,2H,i-4)^\circ. \]
We consider the closures for odd $i$. Namely,
\[ \begin{array}{lcr}
Z_1 \coloneqq \overline{\BN^0_1(B^\circ)} \subset M_H(0,2H,-3) & \text{and} & Z_3 \coloneqq \overline{\BN^0_3(B^\circ)} \subset M \coloneqq M_H(0,2H,-1). 
\end{array} \]
As $M_H(0,2H,-3)$ and $M$ are isomorphic (Lemma \ref{iso mod 2}), $Z_1$ can also be seen as subvarieties in $M$. 

\begin{prop}[Prop \ref{Z alg koisotrop}]
The subvarieties $Z_i \subset M, i=1,3$ have codimension $5-i$ and satisfy
\[ Z_i \subset F_iM.\]
In particular, they are algebraically coisotropic.
\end{prop}

\subsection*{Outline}
In Section \ref{section Mukai 2}, we collect general results on the Mukai system and describe the nature of its fibers. This requires an analysis of the singular curves in $|2H|$. In Section \ref{section orbits}, we state Theorems \ref{theo1} and \ref{theo2} in more detail and apply them to $M= M_H(0,2H,-1)$. Section \ref{section coisotropic} is devoted to present explicit examples. These include the examples from Brill--Noether theory (Section \ref{section horizontal BN}), the examples from singular curves together with their isotropic fibrations (Section \ref{section vertical}) and a less conceptual mixture of example of constant cycle Lagrangians and examples in $S^{[5]}$ (Section \ref{section more}).

\subsection*{Acknowledgements}
This work is part of my PhD thesis at the University of Bonn. I'm extremely grateful to Daniel Huybrechts for all his time and good advice. I wish to thank Thorsten Beckmann, Mirko Mauri, Denis Nesterov, Georg Oberdieck, Giulia Sacc\`a and Johannes Schmitt for their help. Particular thanks go to Hsueh-Yung Lin for reading an earlier version of this manuscript.

\section{The Mukai system}\label{section Mukai 2}
Let $(S,H)$ be a polarized K3 surface of genus $2$ such that the linear system $|H|$ contains a smooth irreducible curve, i.e.\ $S$ is a double covering $\pi \colon S \rightarrow \Pbb^2$ ramified over a smooth sextic curve $R \subset \Pbb^2$ and $H = \pi^*\Ocal_{\Pbb^2}(1)$ is primitive. We consider the moduli space $M=M_H(0,2H,s)$ of $H$-Gieseker stable coherent sheaves on $S$ with Mukai vector $v=(0,2H,s)$ where $s \equiv 1$ mod $2$. This is an irreducible holomorphic symplectic variety of dimension $10$, which is birational to $S^{[5]}$. A point in $M_H(0,2H,s)$ corresponds to a stable sheaf $\Ecal$ on $S$ such that $\Ecal$ is pure of dimension one with support in the linear system $|2H|$ and $\chi(\Ecal)=s$. Taking the (Fitting) support defines a Lagrangian fibration
\[ f \colon M_H(0,2H,s) \longrightarrow B  \coloneqq |2H| \cong \Pbb^{5} \]
known as the \emph{Mukai system of rank two and genus two} \cite{Beau}, \cite{Mu}.\\

As tensoring with $\Ocal_S(H)$ induces an isomorphism
\[ \tau_H \colon M_H(0,2H,s) \xrightarrow\sim M_H(0,2H,s+4),\]
it is immediate that the isomorphism class of $M_H(0,2H,s)$ depends only on $s$ modulo 4. The following Lemma shows that actually the isomorphism class is the same for all odd $s$. If $\Pic(S) = \Z \cdot H$ one could also characterize $M_H(0,2H,s)$ for odd $s$ as the unique birational model of $S^{[5]}$ admitting a Lagrangian fibration.

\begin{lem}\label{iso mod 2}
There is an isomorphism
\[ M(0,2H,1) \longrightarrow M(0,2H,-1),\ \Ecal \mapsto \Ecal^\vee \coloneqq \sheafExt^1_{\Ocal_S}(\Ecal,\Ocal_S). \]
In particular, all the moduli spaces $M_H(0,2H,s)$ for odd $s$ are isomorphic.
\end{lem}

\begin{proof}
Every $\Ecal \in M(0,2H,1)$ is pure of dimension one. Therefore, $\sheafExt^i_{\Ocal_S}(\Ecal,\Ocal_S)=0$ for $i \neq 1$ and the natural map
\[ \Ecal \xlongrightarrow \sim \Ecal^{\vee\vee} = \sheafExt^1_{\Ocal_S}(\sheafExt^1_{\Ocal_S}(\Ecal,\Ocal_S),\Ocal_S)\]
is an isomorphism, \cite[Prop 1.1.10]{HL}. Finally, one easily sees that $\Ecal^\vee$ is again $H$-Gieseker stable. 
\end{proof}

In the following, we usually choose $s= -1$ and set
\[ M \coloneqq M_H(0,2H,-1).\]
With this choice of $s$, a stable vector bundle of rank two and degree one on a smooth curve $C \in |H|$ defines a point in $M$.

\subsection{The linear systems \texorpdfstring{$|H|$}{} and \texorpdfstring{$|2H|$}{}}
The geometry of the Mukai system is closely related to the structure of the curves in the linear systems $|H|$ and $|2H|$, which we want to analyze in this section. A curve in the linear system $|H|$ (resp.\ $|2H|$) has geometric genus $2$ (resp.\ $5$). We use the Segre map
$m \colon |H| \times |H| \rightarrow |2H|$ to define the subloci
\begin{equation}\label{curve subloci}
\Delta \coloneqq m(\Delta_{|H|}) \subset \Sigma \coloneqq \im(m) \subset |2H|.
\end{equation}
Then $\Sigma \cong \Sym^2|H|$ is four-dimensional and its generic member is reduced and has two smooth irreducible components in the linear system $|H|$ meeting transversally in two points. The subset $\Delta \cong |H| \cong \Pbb^2$ is the locus of non-reduced curves. If $\rho(S) =1$, then $\Sigma$ is exactly the locus of non-integral curves.\\

Recall that $\pi \colon S \rightarrow |H| \cong \Pbb^2$ is a double covering, which is ramified along a sextic curve $R \subset \Pbb^2$. We have
\[H^0(S,\Ocal_S(kH)) \cong H^0(\Pbb^2,\Ocal_{\Pbb^2}(k))\oplus H^0(\Pbb^2,\Ocal_{\Pbb^2}(k-3)),\]
and so in particular
\[ H^0(S,\Ocal_S(kH)) \cong H^0(\Pbb^2,\Ocal_{\Pbb^2}(k))\ \text{if}\ k =1,2.\]
We conclude that every curve in $|H|$ (resp.\ in $|2H|$) is the pullback of a line $\ell$ (resp.\ a quadric $Q$) in $\Pbb^2$. In particular, every curve in $|H|$ (resp.\ in $|2H|$) has singularities depending on the intersection behavior of the ramification sextic $R$ with $\ell$ (resp.\ $Q$) and has at most two (resp.\ four) irreducible components. For example, let $\ell \subset \Pbb^2$ be a line and $C \coloneqq \pi^{-1}(\ell) \in |H|$. Assume that $C$ is reducible. Then $C$ consists of two irreducible components $C_1$ and $C_2$, each isomorphic to $\Pbb^1$ with $C_1.C_2 = 3$. This is only possible if $\rho(S) \geq 2$. If $S$ is general, then all curves $C \in |H|$ are irreducible.\\

For an ample line bundle $L$ on $S$ and $0 \leq i \leq \tfrac{L^2}{2}+1$, we can consider the closed subvariety
\begin{equation}\label{def severi var}
V(i,|L|) \coloneqq \{ D \in |L| \mid g(\tilde{D}) \leq i \} \subset |L|,
\end{equation}
which is called a (generalized) Severi variety. We have $\dim V(i,|L|) \leq i$ and there are various results about non-emptiness, irreducibility or smoothness of $V(i,|L|)$ in the literature, e.g.\ \cite{DS}. In our situation, one easily gets a description of $V(i,|H|)$ using the geometry of the covering $\pi \colon S \rightarrow \Pbb^2$.

\begin{prop}\label{V}
The varieties $V(i,|H|)$ are non-empty of dimension $i$ for $i=0,1$. Moreover, $V(1,|H|)$ is irreducible and the locus of nodal curves is dense in $V(1,|H|)$.\\
If $(S,H)$ is general. Then $V(1,|H|) \subset |H| \cong \Pbb^2$ is a nodal curve of degree $30$ and $V(0,|H|)$ consists of $324$ points. Any curve in $V(1,|H|)\setminus V(0,|H|)$ is irreducible and has exactly one node or one cusp as singularities. Any curve in  $V(0,|H|)$ has exactly two nodes as singularities.
\end{prop}

\begin{proof}
From the above discussion, we know that  $V(1,|H|)$ is parameterized by the tangents of $R$, i.e.\
\[ V(1,|H|) \cong R^\vee \subset (\Pbb^2)^\vee = |\Ocal_{\Pbb^2}(1)|, \]
and the dual sextic $R^\vee$ has degree $30$ if $R$ is smooth. A curve in $V(1,|H|)$ is nodal if it corresponds to a tangent line that is tangent to $Q$ in exactly one point. Hence, this locus is dense. A general smooth sextic has exactly $324$ bitangents \cite[IV Ex.\ 2.3]{Hart}.
\end{proof}
  
Next, we study the linear system $|2H|$ and define
\begin{equation}\label{defn V_i}
V_i \coloneqq V(i,|2H|)
\end{equation}
for $i=0,\ldots,5$. Recall that $m\colon |H| \times |H| \rightarrow |2H|$ was the map coming from the Segre embedding. We set
\[ \Sigma_{\{i,j\}} \coloneqq m(V(i,|H|)\times V(j,|H|)) \subset V_{i+j}\]
for $0 \leq i \leq j \leq 2$ and
\[ \Delta_1 \coloneqq m(\Delta_{V(1,|H|)}) \subset V_1, \]
i.e.\ $\Sigma_{\{i,j\}} \subset \Sigma$ is the locus of reducible curves, whose components have geometric genus bounded by $i$ and $j$, respectively and $\Delta_1 \subset \Delta$ is the locus of non-reduced curves, with underlying singular curve. We keep writing $\Sigma$ for $\Sigma_{\{2,2\}}$.

\begin{cor}
We have
\[ \dim\Sigma_{\{i,j\}}= i+j\ \text{and}\ \dim\Delta_1 = 1. \]
Moreover, $\Sigma_{\{i,j\}}$ and $\Delta_1$ are irreducible if $i \neq 0$ and $\Sigma_{\{0,j\}}$ has $324$ irreducible components.
\end{cor}  

Finally, we let
\begin{equation*}
\Lambda_i \coloneqq \overline{\{D \in V_i \mid D\ \text{is integral} \} } \subset V_i.
\end{equation*}
Again, if $D = \pi^{-1}(Q)$ for a quadric $Q \subset \Pbb^2$ the singularities of $D$ depend on the intersection of $Q$ and $R$. This way, one sees that $\Lambda_i$ is an irreducible subvariety of dimension $i$ and a general curve in $\Lambda_i$ has exactly $5-i$ nodes as its only singularities. We sum up this result in the following proposition.

\begin{prop}\label{structure of V_i}
The Severi varieties $V_i \subset |2H|$ are non-empty of pure dimension $i$. Their irreducible components correspond to integral, reducible and non-reduced curves, respectively. More precisely, we have
\begin{align*}
V_4 &= {\Lambda_4} \cup \Sigma \\
V_3 &= {\Lambda_3} \cup \Sigma_{\{1,2\}} \\
V_2 &= {\Lambda_2} \cup \Sigma_{\{0,2\}} \cup \Sigma_{\{1,1\}} \cup \Delta \\
V_1 &= {\Lambda_1} \cup \Sigma_{\{0,1\}}  \cup \Delta_1.
\end{align*}
Here, all occurring varieties but $\Sigma_{\{0,2\}}$ and $\Sigma_{\{0,1\}}$ are irreducible.
\hfill $\square$
\end{prop}

Note that $V_4 = \Lambda_4 \cup \Sigma \subset |2H|\cong \Pbb^5$ is the discriminant divisor of $f$. We compute the degree of its components. 

\begin{lem}
We have
\[
\begin{array}{lcr}
\deg[\Sigma] = 3 & \text{and} & \deg[{\Lambda_4}] = 42 .
\end{array}
\]
In particular, the discriminant divisor of $f$ has degree $45$.
\end{lem}

\begin{proof}
The easiest way, to see that $\deg[\Sigma] = 3$ is a geometric argument. Choose 4 points $x_1,\ldots, x_4$ in general position and consider the line $\ell = \{ D \in |2H| \mid x_i \in D\ \text{for all}\ i=1,\ldots 4\}$. There is a unique (resp. no) curve $C \in |H|$ passing through two (resp. three) points in general position. Hence, $\deg \Sigma = \#(\ell \cap \Sigma) =3$, corresponding to the three possible partitions of $x_1,\ldots, x_4$ into pairs of two points. Alternatively, after a choice of coordinates $\Sigma \cong \Sym^2\Pbb^2$ is embedded into $\Pbb^5$ via the map induced by
\begin{equation}\label{coord}
\begin{aligned}
 \Pbb^2 \times \Pbb^2 &\rightarrow \Pbb^5\\
 [x_0:x_1:x_2], [y_0:y_1:y_2] &\mapsto [x_0y_0:x_1y_1:x_2y_2:x_0y_1+x_1y_0:x_0y_2+x_2y_0:x_1y_2+x_2y_1].
 \end{aligned}
 \end{equation}
One checks that the image is cut out by the equation,
\[ f_0f_5^2 + f_1f_4^2 + f_2f_3^2 = 4f_0f_1f_2 +f_3f_4f_5,\]
where the coordinates $f_i$ of $\Pbb^5$ are ordered as in \eqref{coord}.\\
To prove $\deg[{\Lambda_4}] = 45$, we use the computation from \cite[\S 5]{Saw}. Let $\Ccal'= \bigcup_{t \in \Pbb^1}C_t$ be a general pencil of curves in the linear system $B = |2H|$, i.e.\ $\Ccal' = \Ccal \cap (S \times \Pbb^1)$, where $\Ccal \subset S \times B$ is the universal curve and $\Pbb^1 \subset B$ a general line.  Then $\Ccal' \subset S \times \Pbb^1$ is defined by a section $s \in H^0(S \times \Pbb^1, \Ocal_S(2H) \boxtimes \Ocal_{\Pbb^1}(1))$ and
\[ \Ccal'_{\rm sing} \coloneqq V(s \oplus ds) = \bigcup_{t \in \Pbb^1}(C_t)_{\rm sing} \]
is the union of the singular points of $C_t$, where $ds \in H^0(S \times \Pbb^1,\Omega_S(2H) \boxtimes \Ocal_{\Pbb^1}(1))$. We compute
\begin{equation}\label{sawon}
\deg c_3( (\Ocal_S(2H) \boxtimes \Ocal_{\Pbb^1}(1)) \oplus (\Omega_S(2H) \boxtimes \Ocal_{\Pbb^1}(1))) = 48,
\end{equation}
i.e.\ $\Ccal'_{\rm sing}$ consists of 48 points. As a general pencil contains three curves in $\Sigma$ which have two singular points and a generic integral singular curve has exactly one nodal singularity, we conclude $\deg \Lambda_4 = 42$.
\end{proof}

\begin{remark}
In \cite[\S 5]{Saw} the computation \eqref{sawon} serves as a demonstration for a formula of the degree of the discriminant locus of a Lagrangian fibration with `good singular fibers'. An example of such a fibration is the Beauville--Mukai system over a primitive curve class and the discriminant divisor is irreducible of degree $6(n+3)$, where $n$ is the dimension of the base of the fibration. However, in our example the fibers over $\Delta$ are not `good singular fibers' and we find a different result.
\end{remark}

\subsection{Fibers of the Mukai morphism and structure of \texorpdfstring{$M$}{}}\label{section cc fibers}
In this section, we collect some information on the fibers of the Mukai morphism.\\

The moduli space $M=M_H(0,2H,-1)$ contains a dense open subset consisting of the sheaves that are line bundles on their support. The restriction of the Mukai morphism to this locus is smooth \cite[Prop 2.8]{LP} and the image of the restricted morphism is $B \setminus \Delta$ \cite[Lem 3.5.3]{CRS}. In particular, $M_\Sigma \coloneqq f^{-1}(\Sigma)$ contains a dense open subset that parameterizes the push forwards of line bundles, but $M_\Delta \coloneqq f^{-1}(\Delta)$ does not.\\

Following \cite[Proposition 3.7.1]{CRS}, we can give a description of the fibers of the Mukai morphism. To this end, first assume that $\Pic(S) = \Z \cdot H$. The fibers of the Mukai morphism $f \colon M \rightarrow B$ show the following characteristics:
\begin{equation}
f^{-1}(x)=\
\begin{cases}
\text{is reduced and irreducible} & \text{if}\ x\in B\setminus\Sigma\\
\text{is reduced and has two irreducible components} & \text{if}\ x\in \Sigma\setminus\Delta \\
\text{has two irreducible components with multiplicities} & \text{if}\ x\in \Delta.
\end{cases}
\end{equation}
Let us make this more precise for generic points:
\begin{itemize}
\item In the first case, let $x \in B \setminus \Sigma$ correspond to a smooth curve $D$, then $f^{-1}(x) \cong \Pic^3(D)$.
\item In the second case, let $x \in \Sigma \setminus \Delta$ correspond to the union $D = D_1 \cup D_2$ of two smooth curves meeting transversally in two points. Then $f^{-1}(x)$ contains a dense open subset parameterizing line bundles on $D$. The two irreducible components of $f^{-1}(x)$ correspond to line bundles with partial degree $(2,1)$ and $(1,2)$. %$\Lcal$ with $(\deg(\Lcal|_{D_1}), \deg(\Lcal|_{D_2})) =
\item In the third case, let $x \in \Delta$ correspond to a non-reduced curve with smooth underlying curve $C \in |H|$. Then $f^{-1}(x)$ has two non-reduced irreducible components, which we denote as follows
\begin{equation}\label{fiber decomp over delta}
M_{2C} \coloneqq f^{-1}(x)_{\rm red} = M_{2C}^0 \cup M_{2C}^1.
\end{equation}
The first component  $M_{2C}^0$ consists of those sheaves, that are pushed forward from the reduced curve $C$. With its reduced structure it is isomorphic to the moduli space of stable vector bundles of rank two and degree one on $C$. The other component $M_{2C}^1$ is the closure of those sheaves that can not be endowed with an $\Ocal_C$-module structure. All these sheaves fit into a short exact sequence
\begin{equation}\label{ext}
0 \rightarrow i_*(\Lcal(x)\otimes \omega_C^{-1}) \rightarrow \Ecal \rightarrow i_*\Lcal \rightarrow 0,
\end{equation}
where $i \colon C \hookrightarrow S$ is the inclusion, and $\Lcal \in \Pic^1(C)$ is the torsionfree part of $\Ecal|_C$ and $x \in C$ is the support of the torsion part of $\Ecal|_C$. This extension is intrinsically associated to $\Ecal$, for details see \cite{He}.
\end{itemize}

In the case of a K3 surface of higher Picard rank, the general picture remains the same. But due to reducible curves in the linear system $|H|$ or $B\setminus \Sigma$, the fibers could exceptionally have more irreducible components. For example, if $x \in B\setminus\Sigma$ corresponds to a reducible curve with two smooth components, then $f^{-1}(x)$ still contains a dense open subset parameterizing line bundles. However, following \cite[Lem 3.3.2]{CRS} one finds, that the numerical restrictions imposed by the stability now allow partial degrees $(5,-1),(4,0),(3,1),(2,2),(1,3),(0,4),(-1,5)$. Thus, in this case $f^{-1}(x)$ has seven irreducible components.\\

The decomposition \eqref{fiber decomp over delta} also exists globally over the locus of curves $D=2C$ with $C \in |H|$ smooth, which we denote by $\Delta^\circ \subset \Delta$. Here, we have
\[ M_{\Delta^\circ} = f^{-1}(\Delta^\circ)_{\rm red} = M^0_{\Delta^\circ} \cup M^{1}_{\Delta^\circ}, \]
where $M_{\Delta^\circ}^0$ is a relative moduli space of stable vector bundles and $M_{\Delta^\circ}^1$ the closure of its complement, \cite[Proposition 3.7.23]{CRS}. We set
\begin{equation}\label{comp of M_Delta}
M_\Delta^0 = \overline{M^0_{\Delta^\circ}}\ \text{and}\ M_\Delta^{1} = \overline{M^{1}_{\Delta^\circ}}.
\end{equation}

\section{Orbits under rational equivalence}\label{section orbits}
Our strategy to find algebraically coisotropic subvarieties is to single out points whose orbit under rational equivalence has a high dimension. In this section, we explain how this can be done combining results of Voisin and Shen, Yin and Zhao.

\subsection{Preliminaries}\label{subsection preliminaries}
We start by recalling some general definitions. Let $(X,\sigma)$ be an irreducible holomorphic symplectic manifold of dimension $2n$. For a smooth subvariety $Y \subset X$, we let
\[ \Tcal_Y^\perp \coloneqq \ker(\Tcal_X \xrightarrow\sim \Omega_X \twoheadrightarrow \Omega_Y),\]
where the first arrow is given by $\sigma$.
\begin{enumerate}[\rm (i)]
\item A subvariety $Y \subset X$ is a \emph{constant cycle subvariety} \cite{ccc} if all its points are rationally equivalent in $X$. Note that this is the case, if $Y$ contains a dense open subset $U$, such that all points in $U$ are rationally equivalent in $X$. Mumford's theorem \cite{Mum} implies that a constant cycle subvariety $Y$ is isotropic \cite[Cor 1.2]{Voi}, i.e.\
\[  \Tcal_{Z_{\rm reg}} \subset \Tcal_{Z_{\rm reg}}^\perp\ \text{or equivalently}\ \sigma|_{Y_{\rm reg}} = 0.\]
In particular, $\dim Y \leq n$ and if $\dim Y = n$, then $Y$ is a Lagrangian subvariety.
\item A subvariety $Z \subset X$ is \emph{algebraically coisotropic} \cite[Def 0.6]{Voi} if $Z$ is coisotropic (i.e.\ $ \Tcal_{Z_{\rm reg}} \subset \Tcal_{Z_{\rm reg}}^\perp$) and the corresponding foliation is algebraically integrable. For a subvariety of codimension $i$, this is equivalent to the existence of a $2n-2i$-dimensional variety $T$ and a rational surjective map $\phi \colon Z \dashrightarrow T$ such that
\[  \Tcal_{Z_{\rm reg}}^\perp \cong \Tcal_{Z/T} \ \text{(where defined) or equivalently}\ \sigma|_Z = \phi^*\sigma_T\ \text{for some}\ (2,0)\ \text{form}\ \sigma_T\ \text{on}\ T.\]
Actually, $T$ and $\phi$ are unique up to birational equivalence. We call $\phi$ the associated isotropic fibration.% In the following, whenever we write coisotropic, we will mean algebraically coisotropic.

\item For a point $x \in X$, its \emph{orbit under rational equivalence} is
\[ O_x \coloneqq \{ x' \in X \mid [x]  = [x'] \in \CH_0(X) \} \subset X, \]
which is a countable union of closed algebraic subvarieties \cite[Lem 10.7]{Voibook}. Its dimension is defined to be the supremum over the dimensions of its irreducible components.% and gives an increasing filtration on the points of $M$ as follows.
\end{enumerate}
Following \cite[Def 0.2]{Voi}, we set
\[ F_iX \coloneqq \{  x \in X \mid \dim O_x \geq n-i \} \footnote{We reversed the numbering from \cite{Voi}.}. \]
for $i= 0,\dots,n$. This is again a countable union of closed algebraic subvarieties and defines a filtration on the points of $X$
\begin{equation}%\label{filtration on points}
F_0X \subset F_1X \subset \dots \subset F_nX = X.
\end{equation}
From \cite[Thm 1.3]{Voi} it is known that
\begin{equation}\label{dim von F_i}
\dim F_iX \leq n +i
\end{equation}
and conjecturally \cite[Conj 0.4]{Voi} equality holds true. The following theorem says that a component of maximal dimension is algebraically coisotropic.

\begin{theo}[{\cite[Thm 0.7]{Voi}}]\label{coisotropic}
Let $Z \subset X$ be a subvariety of codimension $n-i$ such that $Z \subset F_{i}X$, then $Z$ is algebraically coisotropic and the fibers of the associated isotropic fibration $\phi \colon Z \dashrightarrow T$ are constant cycle subvarieties of dimension $n-i$.
\end{theo}

%We will use this theorem, in order to produce constant cycle subvarieties.

Now, let $X=M_\sigma(v)$ be a smooth, projective moduli space of (Bridgeland-)stable objects in $D^b(S)$. In this situation, we have the following beautiful criterion for rational equivalence.

\begin{theo}[{\cite{MZ},\cite[Conj 0.3]{SYZ}}]\label{syz}
Two points $\Ecal, \Ecal' \in M_\sigma(v)$ satisfy
\[ [\Ecal] = [\Ecal'] \in \CH_0(M_\sigma(v)) \]
if and only if
\[ \ch_2(\Ecal) = \ch_2(\Ecal') \in \CH_0(S).\]
\end{theo}

\begin{rem}
As $\ch_2(\Ecal) = \tfrac{1}{2}c_1(\Ecal)^2 - c_2(\Ecal)$, and $c_1(\Ecal)$ is fixed for all $\Ecal \in M_\sigma(v)$, one could also phrase the theorem using $c_2$.
\end{rem}

In particular, for $\Ecal \in M_\sigma(v)$ we have
\[ O_\Ecal = \{ \Ecal' \in M_\sigma(v) \mid \ch_2(\Ecal') = \ch_2(\Ecal) \in \CH_0(S) \} \subset M_\sigma(v).\]
Using that the union of all constant cycle curves in $S$ is Zariski dense and Theorem \ref{syz} allows one to prove.
\begin{theo}[{\cite[Thm 0.5(i)]{SYZ}}]
For all $0 \leq i \leq n$ there is an algebraically coisotropic subvariety $Z \subset M_\sigma(v)$ of codimension $i$ such that the isotropic fibration $Z \dashrightarrow T$ has generically constant cycle fibers of dimension $i$. In particular,
\[ \dim F_iM_\sigma(v) = n +i, \]
i.e.\ \eqref{dim von F_i} is actually an equality.
\end{theo}

Next, one could ask how the filtration $F_iM_\sigma(v)$ interferes with the second Chern classes. The answer is to consider O'Grady's filtration on $\CH_0(S)$. Let us recall some results about $\CH_0(S)$.\\

In \cite{BV}, Beauville and Voisin prove that any point lying on a rational curve in $S$ determines the same class
\[ c_S \in \CH_0(S), \]
which has the property that the image of the intersection product $\Pic(S) \otimes \Pic(S) \rightarrow \CH_0(S)$ is contained in $\Z \cdot c_S$.
In \cite{OG}, building on this class, now called the \emph{Beauville--Voisin class}, O'Grady introduces an increasing filtration $S_\bullet$ on $\CH_0(S)$,
\[ S_0\CH_0(S) \subset S_1\CH_0(S) \subset \ldots \subset S_i\CH_0(S) \subset \ldots \subset \CH_0(S),\]
where  $S_i\CH_0(S)$ is the union of cycles of the form $[z] + d \cdot c_S$ for some effective zero-cycle $z$ of degree $i$ and $d \in \Z$. In particular, $S_0\CH_0(S) = \Z \cdot c_S$.  O'Grady's filtration $S_\bullet\CH_0(S)$ has several useful properties, \cite[Cor.\ 1.7 and Claim 0.2]{OG}:
\begin{enumerate}[\rm (1)]

\item The filtration is compatible with addition, i.e.\ if $\alpha \in S_i\CH_0(S)$ and $\beta \in S_j\CH_0(S)$, then $\alpha + \beta \in S_{i+j}\CH_0(S)$.
\item Each step of the filtration $S_i\CH_0(S)$ is closed under multiplication with $\Z$, i.e.\ if $\alpha \in S_i\CH_0(S)$ then $m \cdot \alpha \in S_i\CH_0(S)$ for every $m \in \Z$.
\item If $C$ is an irreducible, smooth projective curve and $f \colon C \rightarrow S$. Then
\[f_*\CH_0(C) \subset S_{g(C)}\CH_0(S).\]
\end{enumerate}

\begin{theo}[{\cite[Thm 0.5(ii)]{SYZ}}]\label{filtrations are equal}
Let $M_\sigma(v)$ be a smooth projective moduli space of Bridgeland stable objects in $D^b(S)$ with $\dim M_\sigma(v) = 2n$. If $\Ecal \in M_\sigma(v)$ satisfies $\ch_2(\Ecal) \in S_i\CH_0(S)$, then $\Ecal \in F_iM_\sigma(v)$.
Moreover, if  $M_\sigma(v)$ is birational to the Hilbert scheme $S^{[n]}$, then also the converse implication holds true, i.e. in this case
\[ F_iM_\sigma(v) = \{ \Ecal \in M_\sigma(v) \mid \ch_2(\Ecal) \in  S_i\CH_0(S) \}.\]
\end{theo}

Note that $c_1(\Ecal)^2 \in \CH_0(S)$. Hence, the theorem can equally be the formulated using $c_2$ instead of $\ch_2$.

\begin{proof}
We sketch the proof along the lines of \cite[Proof of Thm 0.5(ii)]{SYZ}, where the first part of the theorem is proven. The case of $S^{[n]}$, i.e.\ that for all  $\xi \in S^{[n]}$
\[ \dim O_\xi \geq n-i \iff [\Supp(\xi)] \in S_i\CH_0(S) \]
is proven in \cite[Thm 1.4]{Voi0}. Note that the implication from right to left is clear.\\
For the general case, let $\Ecal \in M=M_\sigma(v)$. By \cite[Thm 0.1]{SYZ}, we can write \[\ch_2(\Ecal) = [\Supp(\xi)] + d\cdot [c_S] \in \CH_0(S)\] for some $\xi \in S^{[n]}$ and $d \in  \Z$ depending on the degree of $\ch_2(\Ecal)$, which is fixed. After knowing the result for $S^{[n]}$, the theorem translates into the statement
\begin{equation*}
\begin{array}{lr}
\dim O_\xi \geq n-i \Rightarrow \dim O_\Ecal \geq n-i & (\text{resp.}\ \dim O_\xi \geq n-i \Leftrightarrow  \dim O_\Ecal \geq n-i,\ \text{if}\ M\sim_{\rm bir} S^{[n]}).
\end{array}
\end{equation*}
%(resp.\ $\dim O_\xi \geq n-i\ \Leftrightarrow\ \dim O_\Ecal \geq n-i$, in the case that $M_\sigma(v)$ is birational to $S^{[n]}$).
The two orbits can be compared by means of the incidence variety
\[ R = \{(\Ecal,\xi) \in M \times S^{[n]}\mid \ch_2(\Ecal) = [\Supp(\xi)] + d \cdot [c_S] \in \CH_0(S)\}, \]
which is a countable union of Zariski closed subsets in $M \times S^{[n]}$. There exists an irreducible component $R_0 \subset R$ which projects generically finite and surjective to both factors, and hence yields a correspondence between the two orbits. However, in order to compare their dimensions, one needs to know that the components of maximal dimension in every orbit under rational equivalence are dense. This is known for the Hilbert scheme, whence the inclusion $S_i^{\rm SYZ}\CH_0(M) \subset S_i^{\rm V}\CH_0(M)$ always holds.  The reverse inclusion is true if $M$ is birational to $S^{[n]}$ but in general not known.
\end{proof}

\subsection{Orbits under rational equivalence in \texorpdfstring{$M$}{}}
We turn back to our favorite example $M= M_H(0,2H,-1)$ with the goal in mind, to give explicit constructions of constant cycle subvarieties in $M$. The first step is to understand the orbits under rational equivalence in $M$ and the filtration
\[ F_0M \subset F_1M \subset \ldots \subset F_5M = M.\]
Recall that $M$ is birational to $S^{[5]}$ and thus by Theorem \ref{filtrations are equal}, we know
\begin{align*}
F_iM = \{ \Ecal \in M\mid \dim O_\Ecal \geq 5-i \} = \{ \Ecal \in M\mid \ch_2(\Ecal) \in S_{i}\CH_0(S) \}
\end{align*}
and
\[ \dim F_iM = 5 +i \]
for $0 \leq i \leq 5$.
The following is a straightforward computation using the Grothendieck--Riemann--Roch theorem.
\begin{lem} \label{Hilfslemma c_2 für P}

{\begin{enumerate}[\rm (i)]
Let $i \colon D \hookrightarrow S$ be a reduced curve and let $\Fcal$ be a vector bundle on $D$.
\item Assume that $D$ is irreducible and let $\nu \colon \tilde{D} \rightarrow D$  be its normalization. Then
\begin{equation}
\ch_2(i_*\Fcal) = i_*\nu_*c_1(\nu^*\Fcal) - \rk(\Fcal)( \tfrac{1}{2}i_*\nu_*c_1(\omega_{\tilde{D}})- \sum_{p \in D}m_p[p]) \in \CH_0(S),
\end{equation}
where $m_p= \lg(\nu_*\Ocal_{\tilde{D}}/\Ocal_D)_p$.
In particular,
\[\ch_2(i_*\Fcal) \in \im(\CH_0(\tilde{D}) \xrightarrow{i_*\nu_*} \CH_0(S)) \subset S_{g(\tilde{D})}\CH_0(S).\]
\item  Assume that $D= D_1 \cup D_2$ has two irreducible components. Then
\begin{equation}
\ch_2(i_*\Fcal) = \ch_2({i_1}_*\Fcal|_{D_1}) + \ch_2({i_2}_*\Fcal|_{D_2}) -\rk(\Fcal)(D_1.D_2)c_S  \in \CH_0(S),
\end{equation}
were $i_k \colon D_k \hookrightarrow S, k=1,2$ are the inclusions of the components. In particular,
\[\ch_2(i_*\Fcal) \in S_{g(\tilde{D_1})+g(\tilde{D_2})}\CH_0(S).\]
\end{enumerate}}
\hfill $\square$
\end{lem}

\begin{expl} Using Lemma \ref{Hilfslemma c_2 für P} we compute $\ch_2(\Ecal)$ for some cases of stable sheaves $\Ecal$ occuring in $M$:
\begin{enumerate}[\rm (i)]
\item Let $\Ecal \in M$ such that $D = \Supp(\Ecal)$ is smooth, then $\Ecal= i_*\Lcal$, where $i \colon D \hookrightarrow S$ is the inclusion and $\Lcal\in \Pic^3(D)$. We find
\begin{equation}\label{c2 pushforward of linebundle}
\ch_2(\Ecal) = -4c_S + i_*c_1(\Lcal).
\end{equation}

\item Let $\Ecal \in M$ be the pushforward of a line bundle $\Lcal$ on its support $D= \Supp\Ecal$ and assume that $D= D_1 \cup D_2$ has two smooth and connected components. We write $\Ecal = i_*\Lcal$, then
\begin{equation}\label{c2 sigma}
\ch_2(\Ecal) = -4 c_S + {i_1}_*c_1(\Lcal|_{D_1}) + {i_2}_*c_1(\Lcal|_{D_2}),
\end{equation} 
where $i_k \colon D_k \hookrightarrow C, k=1,2$ are the inclusions.

\item Let $\Ecal \in M_{2C}^0$ for a smooth curve $C \in |H|$, i.e.\ $\Supp(\Ecal) = 2C$ and $\Ecal = i_*\Ecal_0$, where $i \colon C \hookrightarrow S$ is the inclusion and $\Ecal_0$ is a vector bundle of rank 2 and degree 1 on $C$. Then
\begin{equation}\label{c2 vector bundle}
\ch_2(\Ecal) = -2c_S + i_*c_1(\Ecal_0).
\end{equation}

\item Let $\Ecal \in M_{2C}^{1}\setminus M_{2C}^0$ for a smooth curve $C \in |H|$, i.e.\ $\Supp(\Ecal) = 2C$ but $\Ecal$ is not pushed forward along the inclusion $i \colon C \hookrightarrow S$. However, $\Ecal$ fits into a short exact sequence
\[ 0 \rightarrow i_*(\Lcal(x) \otimes \omega_C^{-1}) \rightarrow \Ecal \rightarrow i_*\Lcal \rightarrow 0, \]
for some $\Lcal \in \Pic^1(C)$ and $x \in C$. Hence
\begin{align}\label{c2 other component}
\ch_2(\Ecal) = \ch_2(i_*(\Lcal(x) \otimes \omega_C^{-1}))  + \ch_2(i_*\Lcal) = -4c_S + [i(x)] + 2i_*c_1(\Lcal).
\end{align}

\end{enumerate}
\end{expl}

\begin{cor}\label{ch_2 and g(normalization)}
Let $\Ecal \in M$ and let $D =\Supp(\Ecal)$. Then
\[ \ch_2(\Ecal) \in F_{g(\tilde{D})}M, \]
where $g(\tilde{D})$ is the geometric genus of $D$.
\hfill $\square$
\end{cor}

The geometric genus of $D$ is the genus the normalization of $D$ (resp.\ of $D_{\rm red}$) and the sum over the genera of the normalizations of the irreducible components if $D$ is reducible.

\begin{cor}\label{fiber ccL}
The fiber $F = f^{-1}(D)$ over a curve $D \in |2H|$ is a constant cycle Lagrangian if and only if $D$ is a constant cycle curve in $S$. If $D \in \Delta$ this means that the underlying reduced curve is constant cycle.
\end{cor}

\begin{proof} Note that it suffices to consider a dense open subset of $F$, in order to decide whether $F$ is a constant cycle subvariety.
First assume that $i \colon D \hookrightarrow S$ is reduced. Then $F$ contains a dense open subset parameterizing line bundles of fixed degree. In Lemma \ref{Hilfslemma c_2 für P} we saw that the class of $i_*\Lcal$ in $\CH_0(M)$ depends on
\[  \Pic^k(D) \rightarrow \CH_0(S),\ \Lcal \mapsto i_*c_1(\Lcal), \]
which is constant if $D$ is a constant cycle curve. Conversely, assume that $F \subset M$ is a constant cycle subvariety. Then in particular,
\[ i_*c_1(\Ocal_D(kx))= ki_*[x]  \in S_0\CH_0(S) \]
and hence $[x] = c_S$ for all $x \in D$. (We use that $\CH_0(S)$ is torsionfree).\\
If $D = 2C$ is non-reduced, we apply the same argument to the explicit description \eqref{fiber decomp over delta} of the fiber $F$.
\end{proof}

\section{Algebraically coisotropic subvarieties in \texorpdfstring{$M$}{}} \label{section coisotropic}
We give several examples of algebraically coisotropic subvarieties in $M= M_H(0,2H,-1)$.

\subsection{Horizontal examples from Brill--Noether loci}\label{section horizontal BN}
Brill--Noether theory allows one to produce examples of constant cycle subvarieties.

Let $B^\circ \subset B$ be the locus of smooth curves and $\Ccal^\circ \rightarrow B^\circ$ the restricted universal curve. For any $k$, we have an isomorphism
\[ M_H(0,2H,k-4)^\circ  \cong \Pic^k_{\Ccal^\circ / B^\circ},\]
where $M_H(0,2H,k-4)^\circ$ is the preimage of $B^\circ$ under the support map $M_H(0,2H,k-4) \rightarrow B$. For $k=1,3$, we define
\[ \BN^0_k(B^\circ) \coloneqq \{ \Lcal \in M_H(0,2H,k-4)^\circ\mid H^0(S,\Lcal) \neq 0\} \subset M_H(0,2H,k-4)^\circ \]
and consider the closures
\begin{equation}\label{definition Z}
\begin{array}{lcr}
Z_1 \coloneqq \overline{\BN^0_1(B^\circ)} \subset M_H(0,2H,-3) & \text{and} & Z_3 \coloneqq \overline{\BN^0_3(B^\circ)} \subset M.
\end{array}
\end{equation}
Note that $Z_1$ is strictly contained in $\BN^0(M) \coloneqq \{ \Ecal \in M \mid H^0(S,\Ecal) \neq 0\}$ as $\BN^0(M)$ has a component over $\Delta$.

\begin{prop}\label{Z alg koisotrop}
The subvarieties $Z_i \subset M_H(0,2H,i-4)$ have codimension $5-i$ for $i=1,3$ and satisfy
\[ Z_i \subset F_iM_H(0,2H,i-4). \]
In particular, they are algebraically coisotropic.
\end{prop}

\begin{proof}
A point in $\BN_i^0(B^\circ)$ is of the form $\Ecal = i_*\Ocal_D(\xi)$, where $D \in B^\circ$ and $\xi \subset D$ is an effective divisor of degree $i$. Hence
\[ \ch_2(\Ecal) \equiv [\Supp(\xi)]\ \text{mod}\ \Z\cdot c_S\]
in $\CH_0(S)$ and we conclude $\ch_2(\Ecal) \in S_i\CH_0(S)$, which in turn gives $Z_i \subset F_iM_H(0,2H,i-4)$. 
By Theorem \ref{dim von F_i}, this implies $\dim Z_i \leq 5 +i$, whereas the reverse inequality is known from Brill--Noether theory \cite[IV Lem 3.3]{ACGH}. 
\end{proof}

Actually, $Z_1$ is a projective bundle over $S$. Precisely, let $D \in |2H|$ and $\Lcal \in Z_1 \cap f^{-1}(D)$, i.e.\ $\Lcal \in \Pic^1(D)$ is effective and can uniquely be written as $\Ocal_D(x)$ for some $x \in D$. This way, $Z_1$ is  isomorphic the universal curve $\Ccal \subset |2H| \times S$, which is a $\Pbb^4$ bundle with respect to the second projection. With the same arguments, we also prove that $Z_3$ is generically a $\Pbb^2$-bundle over $S^{[3]}$, which parameterizes the line bundles $\Ocal_D(\xi)$ over $\xi \in S^{[3]}$.\\

In the following, we will consider $Z_1$ as a subvariety of $M$ via the isomorphism
\begin{equation}\label{iso delta}
M_H(0,2H,-3) \rightarrow M,\ \Ecal \mapsto \sheafExt^1(\Ecal,\Ocal_S)\otimes \Ocal_S(-H).
\end{equation}
In particular, over a smooth curve $D \in |2H|$, we have
\[ \Pic^1(D) \rightarrow \Pic^3(D),\ \Lcal \mapsto \Lcal^\vee \otimes \Ocal_S(H)|_D.\]

\begin{lem}\label{lemma z1 in z3}
We have
\[ Z_1 \subset \{ \Ecal \in M \mid h^0(\Ecal)\geq 2\}.\]
In particular, there is an inclusion
\[ Z_1 \subset Z_3. \]
\end{lem}

\begin{proof}
%$\delta(\BN_1^0(B^\circ)) \subset \BN_3^1(B^\circ)$, i.e.\ we want to show that
It suffices to show the result over a smooth curve $D \in |2H|$. Let $\Lcal \in \Pic^1(D)$ such that $H^0(D,\Lcal) \neq 0$. We want to show that $\dim H^0(D,\Lcal^\vee \otimes \Ocal_S(H)|_D) \geq 2$. Write $\Lcal = \Ocal_D(x)$ for a point $x \in D$. On $S$, we have a short exact sequence
\begin{equation*}
0 \rightarrow \Ocal_S(-H) \rightarrow \Ical_x(H) \rightarrow \Ocal_D(-x) \otimes \Ocal_S(H)|_D \rightarrow 0
\end{equation*}
and the resulting long exact cohomology sequence proves the lemma.
\end{proof}

\begin{rem}
One can also define $Z_1$ directly as a subvariety of $M$. Namely, $Z_1$ is the closure of the Brill--Noether locus
\[ \BN^1_3(B^\circ) \coloneqq \{ \Lcal \in M^\circ\mid h^0(S,\Lcal) \geq 2\} \subset M, \]
see \cite{He2}.
Due to the non-primitivity of the linear system $|2H|$, unexpected things happen here. Namely, the smooth curves $D \in |2H|$ are hyperelliptic and we have $W^1_3(D)\neq \emptyset$ for all irreducible curves $D \in |2H|$, even though the Brill--Noether number
$\rho(5,1,3) = 5-2(5-3+1)$
is negative.
\end{rem}

\subsection{Vertical examples from singular curves}\label{section vertical}
In this section, we give examples of algebraically coisotropic subvarieties, that arise as preimages of subvarieties in $B$. In Corollary \ref{fiber ccL}, we already treated the case of a fiber over a point $D \in B$. Namely, $f^{-1}(D)$ is a constant cycle Lagrangian, if and only if $D$ is a constant cycle curve.\\

We set
\[ M_{V_i} \coloneqq f^{-1}(V_i) \subset F_iM,\]
where $V_i  \coloneqq \{ D \in |2H| \mid g(\tilde{D}) \leq i \}$ for $i=1,\ldots,4$ was defined in \eqref{defn V_i}. 
\begin{prop}\label{vertical examples from sing curves}
The subvarieties $M_{V_i}$ are equidimensional of codimension $5-i$ for $i=1,\ldots,4$ and satisfy
\[ M_{V_i} \coloneqq f^{-1}(V_i) \subset F_iM.\]
In particular, they are algebraically coisotropic.
\end{prop}

\begin{proof}
We saw in Proposition \ref{structure of V_i} that $\dim V_i = i$ and in Corollary \ref{ch_2 and g(normalization)} that $g(\tilde{D})\leq i$ implies that $f^{-1}(D) \subset F_iM$ for every $D \in |2H|$.
\end{proof}

In the following section, we find the isotropic fibrations for $M_{V_i}$.

\subsubsection{Isotropic fibrations}
In order to understand the constant cycle subvarieties resulting from the above examples, we write down the isotropic fibration for $\Sigma \subset V_4$ and $\Delta \subset V_2$ and $\Lambda_i \subset V_i$ for $i=1,\ldots,4$.

\begin{prop}\label{T Lambda i}
For every $i= 1,\ldots, 4$, there is a quasi-projective scheme $T_i$ of dimension $2i$ fitting into a diagram
\begin{equation*}
\xymatrix{
M_{{\Lambda_i}} \ar@{-->>}[rr]^{\phi_i} \ar@{>>}[dr]_f& & T_i\ar[dl]\\
&{\Lambda_i}.
}
\end{equation*}
The fibers of $\phi_i$ are rational constant cycle subvarieties of $M$ of dimension $5-i$.
\end{prop}
\begin{proof}
%The existence of $\phi_i \colon M_{\Lambda_i} \dashrightarrow T_i$ is due to Theorem \ref{coisotropic}, but we want to give a constructive proof here.
A general point in $M_{\Lambda_i}$ is the pushforward of a line bundle on a singular curve in $\Lambda_i$. Its class in $\CH_0(S)$ however, depends only on the pullback of the line bundle to the normalization (cf.\ Lemma \ref{Hilfslemma c_2 für P}). This is what $M_{\Lambda_i} \dashrightarrow T_i$ encodes.\\
Consider the universal curve over $|2H|$ and let $\Ccal_i \rightarrow \Lambda_i$ be its restriction to $\Lambda_i \subset |2H|$. By construction, the generic fiber of $\Ccal_i$ is singular and so must be the total space $\Ccal_i$. Hence, the normalization
\[ \tilde{\Ccal_i} \rightarrow \Ccal_i \]
%Because normalization commutes with \'etale base change \comment{ overkill}, there is a dense open subset $U_i \subset \Lambda_i$ such that $\tilde{\Ccal_i}$ has smooth fibers other $U_i$, i.e. $\tilde{\Ccal_i}$ 
generically parameterizes the normalization of the curves in $\Lambda_i$. We set
\[ T_i \coloneqq \Pic^3_{\tilde{\Ccal_i}/U_i}.\]
Then pulling back along $\tilde{\Ccal_i} \rightarrow \Ccal_i$ defines
\[ \phi_i \colon M_{\Lambda_i} \supset \Pic^{3}_{\Ccal_i/\Lambda_i} \dashrightarrow T_i \coloneqq\Pic^3_{\tilde{\Ccal_i}/U_i} \]
and by Lemma \ref{Hilfslemma c_2 für P}(i) the fibers are constant cycle subvarieties of $M$.
Over the open dense subset of curves in $ \Lambda_i$, that have exactly $5-i$ nodes as their only singularities, the fibers of $\phi_i$ are isomorphic to $\Gbb_m^{5-i}$.
\end{proof}

\begin{prop}\label{T sigma}
There is an eight-dimensional quasi-projective scheme $T_\Sigma$ fitting into a diagram
\begin{equation*}
\xymatrix{
M_{\Sigma} \ar@{-->>}[rr]^{\phi_\Sigma} \ar@{>>}[dr]_f& & T_\Sigma\ar[dl].\\
&{\Sigma}.
}
\end{equation*}
The fibers of $\phi_\Sigma$ are rational constant cycle curves in $M$.
\end{prop}

\begin{proof}
%Again, the existence of $\phi_\Sigma \colon M_{Z_\Sigma} \dashrightarrow T_\Sigma$ is due to Theorem \ref{coisotropic}, which we want to make explicit.
A general point in $M_\Sigma$ is the pushforward of a line bundle on a reducible curve $i\colon D \hookrightarrow S$ and by Lemma \ref{Hilfslemma c_2 für P} (ii) the class $[i_*\Lcal] \in \CH_0(S)$ depends exactly on the restriction of $\Lcal$ to each component. This is, what $T_\Sigma$ shall parameterize.\\
Let $\Ccal_{\Sigma\setminus\Delta} \rightarrow \Sigma\setminus \Delta$ be the universal curve over $\Sigma\setminus\Delta$. Even though every fiber has two irreducible components, the total space $\Ccal_{\Sigma\setminus\Delta}$ is irreducible. However, after the base change
\begin{equation*}
\xymatrix{ \tilde{\Ccal}_{\Sigma\setminus\Delta} \ar[r]\ar[d] & \Ccal_{\Sigma\setminus\Delta} \ar[d] \\
\Pbb^2 \times \Pbb^2 \setminus \Delta \ar[r] & \Sigma\setminus\Delta,
\
}
\end{equation*}
we have a decomposition $\tilde{\Ccal}_{\Sigma\setminus\Delta} = \tilde{\Ccal}^1_{\Sigma\setminus\Delta} \cup \tilde{\Ccal}^2_{\Sigma\setminus\Delta}$
into two irreducible components, which are identified under the natural $\Z/2\Z$-action. Note that the horizontal arrows are principal $\Z/2\Z$-bundles and the vertical arrows are $\Z/2\Z$-equivariant. On the level of Picard schemes, restricting to each component gives a $\Z/2\Z$-equivariant map
\begin{equation}\label{restrict}
 \Pic_{\tilde{\Ccal}_{\Sigma\setminus\Delta}/{\Sigma\setminus\Delta}} \longrightarrow \Pic_{\tilde{\Ccal}^1_{\Sigma\setminus\Delta}/{\Sigma\setminus\Delta}} \times_{\Sigma\setminus\Delta}\Pic_{\tilde{\Ccal}^2_{\Sigma\setminus\Delta}/{\Sigma\setminus\Delta}}.
 \end{equation}
Here, $\Z/2\Z = \langle \tau\rangle$ acts on the right hand side via
\[\tau \cdot(\Lcal_1,\Lcal_2) = (\tau^*\Lcal_2,\tau^*\Lcal_1).\]
(By a slight abuse of notation, $\tau$ also denotes the map identifying the isomorphic components $\tilde{\Ccal}^1_{\Sigma\setminus\Delta}$ and $\tilde{\Ccal}^2_{\Sigma\setminus\Delta}$). The quotient 
\[ \Pic_{\tilde{\Ccal}_{\Sigma\setminus\Delta}/{\Sigma\setminus\Delta}}/\tau \cong \Pic_{\Ccal_{\Sigma\setminus\Delta}/\Sigma\setminus\Delta} \longrightarrow (\Pic_{\tilde{\Ccal}^1_{\Sigma\setminus\Delta}/{\Sigma\setminus\Delta}} \times_{\Sigma\setminus\Delta}\Pic_{\tilde{\Ccal}^2_{\Sigma\setminus\Delta}/{\Sigma\setminus\Delta}})/\tau\]
of \eqref{restrict} by $\tau$ is what we are looking for, once the correct degree has been fixed. We set
\[ T_\Sigma \coloneqq (
\Pic^1_{\tilde{\Ccal}^1_{\Sigma\setminus\Delta}/{\Sigma\setminus\Delta}} \times_{\Sigma\setminus\Delta} \Pic^2_{\tilde{\Ccal}^2_{\Sigma\setminus\Delta}/{\Sigma\setminus\Delta}}
\sqcup
\Pic^2_{\tilde{\Ccal}^1_{\Sigma\setminus\Delta}/{\Sigma\setminus\Delta}} \times_{\Sigma\setminus\Delta}\Pic^1_{\tilde{\Ccal}^2_{\Sigma\setminus\Delta}/{\Sigma\setminus\Delta}}
)/\tau \]
and take $\phi_\Sigma$ to be the above map, whose fibers are isomorphic to $\Gbb_m$.
\end{proof}

\begin{prop}\label{T Delta}
There is a four-dimensional quasi-projective scheme $T_\Delta$ fitting into a diagram
\begin{equation*}
\xymatrix{
M^i_{\Delta} \ar@{-->>}[rr]^{\phi^i_\Delta} \ar@{>>}[dr]_f& & T_\Delta\ar[dl]\\
&{\Delta}
}
\end{equation*}
for $i=0,1$. The fibers of $\phi^0_\Delta$ are three-dimensional rational constant cycle subvarieties in $M$. Over $2C \in \Delta$, the fibers of $(\phi^1_\Delta)_{2C}$ are birational to a $\Pbb^2$-bundles over the disjoint union of $16$ copies of $C$. In particular, they yield examples of three-dimensional constant cycle subvarieties in $M$ that are not rational.
\end{prop}

\begin{proof}
We consider the component $M_\Delta^0$ first. A general point in $M^0_\Delta$ is of the form $\Ecal = i_*\Ecal_0$, where $i\colon C \hookrightarrow S$ is the inclusion of a smooth curve $C \in |H|\cong \Delta$ and $\Ecal_0$ is a vector bundle of rank 2 on $C$. The class $[i_*\Ecal_0] \in \CH_0(S)$ is determined by $i_*c_1(\Ecal_0)$. This suggests to set
\[ T_\Delta \coloneqq \Pic^1_{\Ccal_U/U},\]
where $U \subset |H|$ is the ope subset of smooth curves and $\Ccal_U \rightarrow U$ denotes the universal curve and then define $\phi^0_\Delta$ as the determinant map. The fibers of $\phi^0_\Delta$ are isomorphic to a moduli space of stable vector bundles of rank two with fixed determinant of degree one, which is rational \cite{New}, \cite[Prop 2]{NewCor}.\\
%More precisely, recall that $U \subset |H|$ is the locus of smooth curves and $\Ccal \rightarrow U$ denotes the universal curve. Moreover, $M_\Delta^0$ contains the relative moduli space $M_{\Ccal/U}(2,1) \rightarrow U$ as a dense open subset. We define
%\[ \phi_^0_\Delta \colon M_\Delta^0 \dashrightarrow T^0_\Delta \coloneqq \Pic^1_{\Ccal/U}\]
%to be the determinant map
To deal with $M_\Delta^1$, let $\Ecal \in M_\Delta^1\setminus (M_{\Delta}^0 \cap M_{\Delta}^1)$ such that $C \coloneqq \Supp(\Ecal)_{\rm red} \in U$. Then, by \eqref{ext} and Lemma \ref{Hilfslemma c_2 für P} the class $\ch_2(\Ecal)$ is determined by $i_*c_1(\Lcal^{\otimes 2}(x)\otimes \omega_C^{-1})$, where $\Lcal \coloneqq  \Ecal|_C/\Tcal \in \Pic^1(C)$ and $x \coloneqq \Supp(\Tcal) \in C$ with $\Tcal$ being the torsion subsheaf of $\Ecal|_C$ and $i \colon C \hookrightarrow S$ being the inclusion. Consequently, we define
\[ \phi_\Delta^1 \colon M_\Delta^1 \dashrightarrow T_\Delta, \Ecal \mapsto \Lcal^{\otimes 2}(x)\otimes \omega_C^{-1}.\]
We want to compute the fibers of $(\phi_\Delta^1)_{2C}$. First, we forget the twist with $\omega_C^{-1}$. Then we can factor $(\phi_\Delta^1)_{2C}$ as follows
\[ M_{2C}^1 \dashrightarrow \Pic^1(C) \times C \rightarrow \Pic^3(C),\ \Ecal \mapsto (\Lcal,x) \mapsto \Lcal^{\otimes 2}(x). \]
%\begin{align*}
%M_{2C}^1 &\rightarrow \Pic^1(C) \times C \rightarrow \Pic^3(C) \xrightarrow\sim \Pic^1(C)\\
%\Ecal &\mapsto (\Lcal,x) \mapsto \Lcal^{\otimes 2}(x) \mapsto \Lcal^{\otimes 2}(x)\otimes \omega_C^{-1}.
%\end{align*}
The first arrow is defined outside the intersection $M_{2C}^0 \cap M_{2C}^1$ and its fibers are a torsor under $\Ext^1_C(i_*\Lcal,i_*(\Lcal(x)\otimes \omega_C^{-1})) \cong \C^2$, cf.\ \cite[Cor 3.5]{He}. Thus the fibers of $(\phi_\Delta^1)_{2C}$ are an $\Abb^2$-bundle over the fibers of the second arrow. We claim that the latter are isomorphic to $16$ copies of $C$. This claim follows from the factorization
\[ \Pic^1(C) \times C \rightarrow \Pic^2(C) \times C \xrightarrow\mu \Pic^3(C),\ (\Lcal,x)\mapsto (\Lcal^{\otimes  2}, x) \mapsto \Lcal^{\otimes 2}(x),\]
where the first map is \'etale of degree 16 and the fibers of the second map
\[ \mu \colon \Pic^2(C)\times C \rightarrow \Pic^3(C),\ (\Lcal,x)\mapsto \Lcal(x)\]
are isomorphic to $C$. To see this, let $\Mcal \in \Pic^3(C)$ and consider $p_2\colon \mu^{-1}(\Mcal) \rightarrow C$. As $\Lcal(x) \cong \Mcal$ for fixed $x \in C$ determines $\Lcal \in \Pic^2(C)$, this projection is an isomorphism and the claim follows.
\end{proof}

\begin{rem}
A combination of the proofs of Propositions \ref{T Lambda i}, \ref{T sigma} and \ref{T Delta} allows one to find the isotropic fibrations for the remaining components of $V_i, i=1,2,3$. 
\end{rem}

\subsection{More examples}\label{section more}
We construct some more examples of algebraically coisotropic subvarieties.

\subsubsection{Horizontal constant cycle Lagrangians}
To start with, we produce a constant cycle Lagrangian that dominates $B$. For example, any section of $M \rightarrow B$ would work. Unfortunately, $f$ does not admit a section \cite{Bakker}. Below, we produce a multisection of degree of $2^{10}$. Recall that
\[ M^\circ  \cong \Pic^3_{\Ccal^\circ / B^\circ}\]
and there is an exact sequence \cite[(9.2.11.5)]{FGA}
\[ 0 \rightarrow \Pic(\Ccal^\circ)/\Pic(B^\circ)  \rightarrow \Pic_{\Ccal^\circ / B^\circ}(B^\circ) \rightarrow \Br(B^\circ) \rightarrow \ldots. \]
Moreover, one can show that
\[  \Pic(\Ccal^\circ)/\Pic(B^\circ) \cong \Pic(\Ccal)/\Pic(B) \cong \Pic(S), \]
where the last isomorphism holds because $\Ccal \subset B \times S$ is a $\Pbb^4$-bundle over $S$ with $\Ocal_{p_2}(1) = p_1^*\Ocal_B(1)$. For $L \in \Pic(S)$ with $n= 2H.L$, the corresponding section is given by
\[ s_L \colon B^\circ \rightarrow \Pic^n_{\Ccal^\circ / B^\circ},\ D \mapsto L.D.\]
If $\Pic(S)= \Z \cdot H$, for example, one gets sections for $n \equiv 0$ mod $4$. There is always a section of $\Pic^2_{\Ccal^\circ / B^\circ}$ that does not come from $S$.

\begin{lem}\label{g^1_2}
There is a section
\[ g^1_2 \colon B^\circ \rightarrow \Pic^2_{\Ccal^\circ / B^\circ}\]
such that a curve $D \in B^\circ$ maps to the unique line bundle $g^1_2(D) \in \Pic^2(D)$ with $h^0(g^1_2(D))=2$. In particular,
\[ (g^1_2) \otimes (g^1_2) = s_H \colon B^\circ \rightarrow \Pic^4_{\Ccal^\circ / B^\circ}\]
and $g^1_2$ is not of the form $s_L$ for $L\in \Pic(S)$.
\end{lem}

\begin{proof}
This is a consequence of the same phenomenon occurring for the universal family of smooth quadrics in $\Pbb^2$. We identify $B = |\Ocal_{\Pbb^2}(2)|$ and we will see that the lemma holds true for $B^\circ = B \setminus \Sigma$. Let $\Qcal^\circ \subset B^\circ \times \Pbb^2$ be the universal quadric, which is an \'etale $\Pbb^1$-fibration, but not a projective bundle. We claim that there is a section
\[ s \colon B^\circ \rightarrow \Pic^1_{\Qcal^\circ/B^\circ}\]
that is not induced by a line bundle on $\Qcal^\circ$.  Indeed, fix a line $\ell \subset \Pbb^2$ and consider
\[ \tilde{B^\circ} \coloneqq \Qcal^\circ \cap (B^\circ \times \ell) \rightarrow B^\circ.\]
This morphism is finite, flat of degree $2$ and the base change $\tilde{\Qcal^\circ}\rightarrow \tilde{B^\circ}$ admits a section. Therefore we get
\[ \tilde{s} \colon \tilde{B^\circ} \rightarrow \Pic^1_{\tilde{\Qcal^\circ}/\tilde{B^\circ}}.\]
As the two points  in $\tilde{B^\circ}$ lying over a fixed point in $B^\circ$ define the same line bundle, $\tilde{s}$ descends to a section $s$. By definition, $s\otimes s$ is the section defined by $p_2^*\Ocal_{\Pbb^2}(1)$, which does not admit a square root. Pulling back $s$ along $\Ccal^\circ \rightarrow \Qcal^\circ$ defines $g^1_2$.% The last statement follows because $H \in \Pic(S)$ is primitive.
\end{proof}

\begin{rem}
In all our examples, it does not matter if we identify $M$ and $M_H(0,2H,-3)$ via the isomorphism \eqref{iso delta} (given by tensorization and dualization) or the birational map induced by the section $g^1_2$. The composition of the one map with the inverse of the other is the rational involution on $M$ that comes from the natural involution $\iota^{[5]}$ on $S^{[5]}$.
\end{rem}

Now, we use the squaring map
\[ \rho_2 \colon \Pic^1_{\Ccal^\circ / B^\circ} \longrightarrow \Pic^2_{\Ccal^\circ / B^\circ},\ \Lcal \mapsto \Lcal^{\otimes 2}, \]
to construct a constant cycle Lagrangian from $g^1_2(B^\circ)$. Specifically, we set
\[ L^1_2 \coloneqq \overline{\rho_2^{-1}(g^1_2(B^\circ))} \subset M_H(0,2H,-3) \xrightarrow\sim M.\] 

\begin{lem}\label{L12}
The subvariety $L^1_2$ is a constant cycle Lagrangian in $M$, which is generically finite of degree $2^{10}$ over $B$.
\end{lem}
\begin{proof}
It is clear, that $\dim L^1_2 = 5$. We will show that $L^1_2 \subset F_0 M_H(0,2H,-3)$. Let $D \in B^\circ$ and $\Lcal \in \Pic^1(D)$ such that $\Lcal^{\otimes 2}= g^1_2(D)$. Then
\[ 4 \cdot i_*c_1(\Lcal) = i_*c_1(\Ocal_S(H)|_D) =4 c_S \in \CH_0(S).\]
This implies $i_*c_1(\Lcal) = c_S$ because $\CH_0(S)$ is torsionfree. We conclude that a general point in $L^1_2$ is contained in $S_0\CH_0(S)$ as desired. Finally, $\rho_2$ is finite, \'etale of degree $2^{10}$.
\end{proof}
%\comment{ $\pi_1(B^\circ) = \Z \oplus \Z/3$ is falsch}

\begin{rem}
Another example of a horizontal constant cycle Lagrangian is constructed more generally for any Lagrangian fibration by Lin \cite{Lin}.
\end{rem}

\subsubsection{Examples in $M_\Delta$}
Starting from $\phi_\Delta^i \colon M_\Delta^i \dashrightarrow T_\Delta$ (cf.\ Proposition \ref{T Delta}), we construct two examples in $M_\Delta^0$ and $M_\Delta^1$. Recall that $T_\Delta = \Pic^1_{\Ccal_U/U}$, where $U \subset |H|$ is the open subset consisting of smooth curves and the fibers of $\phi_\Delta^i$ are three-dimensional constant cycle subvarieties in $M$.

The idea of Example \ref{exa 1} is to find a constant cycle surface in $T_\Delta$. Then the preimage under $\phi^i_\Delta$ is a constant cycle Lagrangian in $M$ contained in $M_\Delta^i$. This idea is taken further in Example \ref{exa 2}. Here, we find a surface in $T_\Delta$, that consists of line bundles whose first Chern class is a multiple of the Beauville--Voisin class, when pushed forward to $S$. By Theorem \ref{filtrations are equal}, the preimage of this surface is also a constant cycle Lagrangian in $M$.\\
For simplicity, we assume from now on that $\Pic(S) = \Z \cdot H$. Then every curve in $|H|$ is integral and $\Pic^1_{\Ccal_{|H|}/|H|}$ is representable by a smooth, quasi-projective scheme.

\begin{expl}\label{exa 1}
We construct a constant cycle subvariety of $\Pic^1_{\Ccal_{|H|}/|H|}$ applying the same trick as for the construction of $L^{1}_2$. Namely, let
\[ \rho_2 \colon \Pic^1_{\Ccal_{|H|}/|H|} \rightarrow \Pic^2_{\Ccal_{|H|}/|H|}\]
and consider the section $s_H$ of $\Pic^2_{\Ccal_{|H|}/|H|}$ defined by $H$. We set
\[ Z_H \coloneqq \rho_2^{-1}(s_H(|H|)).\]
Since $\Pic^2_{\Ccal_{|H|}/|H|}$ can be embedded as an open subset of $M_H(0,H,2)$, we can apply Theorem \ref{filtrations are equal} to see that $Z_H$ is a constant cycle subvariety, as in the proof of Lemma \ref{L12}.
Now, $Z_H \subset \Pic^2_{\Ccal_{|H|}/|H|}$ is a smooth, quasi-projective surface and the morphism $Z_H \rightarrow |H|$ is finite, \'etale of degree $2^4$, when restricted to the open subset of smooth curves $U \subset |H|$.  By Lemma \ref{V}, we know that $U= |H| \setminus V(1,|H|)$ is the complement of a nodal curve of degree $30$. Therefore $\pi_1(U) \cong \Z/30\Z$ \cite[Prop 1.3 \& Thm 1.13]{Dimca}. Consequently, $Z_H$ must have $8$ pairwise isomorphic connected components that restrict over $U$ to the unique degree 2 cover of $U$. We replace $Z_H$ by one of its irreducible components and define
\[ L^i \coloneqq \overline{(\phi^i_\Delta)^{-1}(Z_H)} \subset M_\Delta^i\ \text{for}\ i=0,1.\]
By construction, these are constant cycle Lagrangians in $M$.
\end{expl}

\begin{expl}\label{exa 2}
The idea of this example is to consider the preimage of two-dimesional subvarieties in $T_\Delta$ that are not constant cycle subvarieties themselves, but consist of line bundles whose first Chern class is the Beauville--Voisin class when pushed forward to $S$.\\
To begin with, we have an embedding
\[  \Theta \colon \Ccal_{|H|} \hookrightarrow \Pic^1_{\Ccal_{|H|}/|H|},\ C \ni x \mapsto \Ocal_C(x).\]
Then, for example a vector bundle $\Ecal \in M^0_\Delta$ lies over $\Theta(\Ccal_{|H|})$ if and only if its determinant line bundle is effective (of degree one). Therefore,
\[ \begin{array}{lcr}
(\phi^i_\Delta)^{-1}(\Theta(\Ccal_{|H|})) \subset F_1M & \text{and} &\codim (\phi^i_\Delta)^{-1}(\Theta(\Ccal_{|H|})) =4.
\end{array}\]
In particular, $(\phi^i_\Delta)^{-1}(\Theta(\Ccal_{|H|}))$ is algebraically coisotropic. The isotropic fibration is given by the composition of the projection $\Ccal \rightarrow S$ with $\phi^i_\Delta$.\\
Refining this example yields constant cycle Lagrangians in $M_\Delta^i$ as follows. For example, let $C_{\rm cc} \subset S$ be a constant cycle curve and set
\[ L_{C_{\rm cc}} \coloneqq (\phi^i_\Delta)^{-1}(\Theta(\Ccal_{|H|} \cap C_{\rm cc} \times |H|)).\]

\end{expl}

\subsubsection{Examples in $S^{[5]}$}\label{subsection ex in S5}
We can also produce easily examples of algebraically coisotropic subvarieties in $S^{[5]}$.
As $M$ and $S^{[5]}$ are birational, we have
\[ \CH_0(S^{[5]}) \cong \CH_0(M), \]
\cite[Expl 16.1.11]{Fulton} and algebraically coisotropic varieties that are not contained in the exceptional locus of a birational map can be transferred from $S^{[5]}$ to $M$ and vice versa.\\

\begin{example}\label{E}
This example can also be found in \cite[\S 4 Exa 1]{Voi}. For $i=1,\ldots,4$, define
\[ E_i \coloneqq \{ \xi \in S^{[5]} \mid \lg(\Ocal_{\xi_{\rm red}}) \leq i\  \}. \]
%\[ E_i \coloneqq \{ \xi \in S^{[5]}|\ \xi_{\rm red} \ \text{consists of}\ \leq i\ \text{points}\ \}. \]
Then $E_i \subset S^{[5]}$ is closed subvariety of codimension $5-i$ \cite{Bri}. For example, $E \coloneqq E_4$ is the exceptional divisor of the Hilbert--Chow morphism $s \colon S^{[5]} \rightarrow S^{(5)}$. The irreducible components $E_i^{\underline{n}}$ of $E_i$ are indexed by ordered tuples of positive natural numbers $\underline{n} = (n_1 \geq n_2 \geq \ldots \geq n_i)$ such that $\sum_{k=1}^i n_k =5$. In particular, $E_4$ and $E_1$ are irreducible, whereas $E_3$ and $E_2$ consists of two irreducible components. To sum up
\begin{equation}\label{chain of E}
%E_1 \subset E_2= (E_2^{4,1} \cup E_2^{3,2}) \subset E_3 = (E_3^{3,1,1} \cup E_3^{2,2,1}) \subset E \subset S^{[5]}.
E_1 \subset E_2 \subset E_3 \subset E_4 = E \subset S^{[5]}.
\end{equation}
By definition of $E_i$ and Theorem \ref{syz}, we have
\[ E_i \subset F_iS^{[5]}\]
for all $i=1,\ldots, 4$ and hence $E_i$ is algebraically coisotropic. 
\end{example}

\begin{example}\label{P}
We have $\Pbb^2 \subset S^{[2]}$ given by $x \mapsto \pi^{-1}(x)$, where $\pi \colon S \rightarrow \Pbb^2$.
Consider the generically injective rational maps
\[ \begin{array}{lcr}
g_3 \colon \Pbb^2 \times S^{[3]} \dashrightarrow S^{[5]} & \text{and} & g_1 \colon \Pbb^2 \times \Pbb^2 \times S \dashrightarrow S^{[5]}
\end{array}\]
and set
\[ P_i \coloneqq \overline{\im(g_i)} \subset S^{[5]}\ \text{for}\ i=1,3.\]
Clearly, $P_i \subset F_iS^{[5]}$ and $\codim P_i = 5-i$.
\end{example}

\begin{example}\label{W}
This example is taken from \cite{KLM}.  As in \eqref{def severi var}, we consider the locus $V(j,|H|) \subset |H|$ of curves $C$ with $g(\tilde{C})\leq j$ for $j=0,1,2$. Specifically, $V(2,|H|)$ is everything, $V(1,|H|) \subset |H|$ is irreducible and 1-dimensional and the generic curve in $V(1,|H|)$ has exactly one node, $V(0,|H|)$ is the discrete set of rational curves. We let $\Ccal_j \rightarrow V(2-j,|H|)$ be the respective restriction of the universal curve $\Ccal_{|H|} \rightarrow |H|$ and $\tilde{C_j}$ its normalization. For $2-j \leq i \leq 4$, consider the diagram
\[ \xymatrix{
& S^{[7-i-j]} \times S^{[i+j-2]} \ar@{-->}[r] \ar[d]^{s\times \id} & S^{[5]}\\
\Sym_{V(2-j,|H|)}^{7-i-j}(\Ccal_j) \times S^{[i+j-2]} \ar@{-->}@/^3,5pc/[rru]^{f^j_i} \ar[r] & S^{(7-i-j)} \times S^{[i+j-2]},
}
\]
%\[ f^j_i \colon \Sym_{V(2-j,|H|)}^{7-i-j}(\Ccal_j) \times S^{[i+j-2]} \dashrightarrow S^{(5)}, \]
in which the lower horizontal map and hence $f^j_i$ turns out to be generically injective \cite[Thm 6.4]{KLM}. %We define $W^j_i$ as the closure of the image under $\mu$ of $\widetilde{\im(f^j_i)}\times S^{[i+j-2]}$, where $\widetilde{\im(f^j_i)}$ is the strict transform of $\im(f^j_i)$ under the Hilbert--Chow morphism.
We define
\[ W^j_i \coloneqq \overline{\im(f^j_i)} \subset S^{[5]}.\]
This is a subvariety of codimension $5-i$, which is irreducible for $j\neq 2$. We have the following table of inclusions:
 \begin{equation*}\label{chain of W}
  \begin{array}{cccccccccc}
W^0_2 & \subset  & W^0_3  &\subset & W^0_4\\
\cup && \cup && \cup \\
W^1_1 & \subset  & W^1_3  &\subset & W^1_3 &\subset & W^1_4  \\
\cup && \cup && \cup && \cup\\
W_0^2 & \subset  & W^2_1  &\subset & W_2^2 &\subset & W^2_3 &\subset & W^2_4. \\
 \end{array}
 \end{equation*}
A generic point $\xi \in W^j_i$ corresponds to a subscheme in $S$ that contains exactly $7-i-j$ points, which lie on a curve in $V(2-j,|H|)$ and the other $i+j-2$ points can move freely outside $C$. Hence, $[\xi]\in \CH_0(S)$ is contained in the $(2-j) + (i+j-2)$-th step of O'Grady's filtration. In other words,
\[ W^j_i \subset F_iS^{[5]}\]
and $W^j_i$ is algebraically coisotropic. The isotropic fibration on $W^j_i$ is given by the Abel map
\[ \Sym_{V(2-j,|H|)}^{7-i-j}(\Ccal_j) \rightarrow \Pic^{7-i-j}_{\tilde{\Ccal_j}/V(2-j,|H|)} \]
and endows $W^j_i$ generically with the structure of a $\Pbb^{5-i}$-bundle. By \cite[Thm 6.4]{KLM} the class of a line in the fibers is $H-(8-i+2j)\delta^\vee$. 
\end{example}

\begin{rem}
One can show that $W_2^0$ is the $\Pbb^3$-bundle over $M_H(0,H,-6)$ parameterizing extensions $\Ext_S^1(\Ecal, \Ocal_S(-H))$. Moreover, $W^0_3\setminus W^0_2$ has the structure of a $\Pbb^2$-bundle over a dense open subset of $S \times M_H(0,H,-5)$. This bundle parameterizes ideal sheaves $\Ical \in M_H(1,0,-4)=S^{[5]}$ that fit into an extension
\[  0 \rightarrow \Ical_x(-H) \rightarrow \Ical \rightarrow \Ecal \rightarrow 0.\]
\end{rem}

\end{document}